\documentclass[10pt, reqno]{amsart}

\usepackage{amsmath,amssymb,amsthm,amsfonts,verbatim}
\usepackage{microtype}
\usepackage[all,2cell]{xy}
\usepackage{mathtools}
\usepackage{graphicx}
\usepackage{pinlabel}
\usepackage{hyperref}
\usepackage{mathrsfs}
\usepackage{color}
\usepackage{enumerate}

\CompileMatrices

\usepackage[top=1.3in,bottom=1.9in,left=1.4in,right=1.4in]{geometry}

\theoremstyle{plain}
\newtheorem{theorem}{Theorem}[section]
\newtheorem{maintheorem}{Theorem}

\newtheorem{proposition}[theorem]{Proposition}
\newtheorem{lemma}[theorem]{Lemma}

\newtheorem{fact}[theorem]{Fact}

\theoremstyle{definition}
\newtheorem{definition}[theorem]{Definition}

\newtheorem{remark}[theorem]{Remark}

\newcommand{\nc}{\newcommand}
\nc{\dmo}{\DeclareMathOperator}

\nc{\Q}{\mathbb{Q}}
\nc{\F}{\mathbb{F}}
\nc{\R}{\mathbb{R}}
\nc{\Z}{\mathbb{Z}}
\nc{\C}{\mathbb{C}}
\nc{\Ell}{\mathcal{L}}
\nc{\M}{\mathcal{M}}
\nc{\K}{\mathcal{K}}
\nc{\I}{\mathcal{I}}
\nc{\U}{\mathcal U}
\nc{\disk}{\mathbb{D}}
\nc{\hyp}{\mathbb{H}}
\renewcommand{\P}{\mathbb{P}}
\nc{\CP}{\mathbb{CP}}
\nc{\cS}{\mathcal{S}}
\dmo{\Mod}{Mod}
\dmo{\Diff}{Diff}
\dmo{\Homeo}{Homeo}
\dmo{\dist}{dist}
\dmo\BDiff{BDiff}
\dmo\SO{SO}
\dmo\Hom{Hom}
\dmo\SL{SL}
\dmo\Sp{Sp}
\dmo\rank{rank}
\dmo\sig{sig}
\dmo\Out{Out}
\dmo\Aut{Aut}
\dmo\Inn{Inn}
\dmo\GL{GL}
\dmo\PSL{PSL}
\dmo\BHomeo{BHomeo}
\dmo\EHomeo{EHomeo}
\dmo\EDiff{EDiff}
\nc\Sig{\Sigma}
\dmo\Teich{Teich}
\dmo\Fix{Fix}
\nc{\pair}[1]{\langle #1 \rangle}
\nc{\abs}[1]{\left| #1 \right|}
\nc{\action}{\circlearrowright}
\nc{\norm}[1]{\left | \left | #1 \right | \right |}
\nc{\abcd}[4]{\left(\begin{array}{cc} #1 & #2 \\ #3 & #4 \end{array}\right)}
\nc{\into}\hookrightarrow
\dmo{\Isom}{Isom}
\nc{\normal}{\vartriangleleft}
\dmo{\Vol}{Vol}
\dmo{\im}{Im}
\dmo{\Push}{Push}
\dmo{\Conf}{Conf}
\dmo{\PConf}{PConf}
\dmo{\id}{id}
\dmo{\Jac}{Jac}
\dmo{\Pic}{Pic}

\renewcommand{\epsilon}{\varepsilon}
\nc{\coloneq}{\mathrel{\mathop:}\mkern-1.2mu=}
\nc{\margin}[1]{\marginpar{\scriptsize #1}}
\nc{\para}[1]{\medskip\noindent\textbf{#1.}}
\nc{\red}[1]{\textcolor{red}{#1}}

\title[Monodromy of plane curves]{On the monodromy group of the family of smooth plane curves}

\author{Nick Salter}
\email{nks@math.uchicago.edu}
\date{October 16, 2016}

\address{Department of Mathematics\\ University of Chicago\\ 5734 S. University Ave., Chicago, IL 60637}

\begin{document}
\maketitle

\begin{abstract}
We consider the space $\mathcal P_d$ of smooth complex projective plane curves of degree $d$. There is the tautological family of plane curves defined over $\mathcal P_d$, and hence there is a monodromy representation $\rho_d: \pi_1(\mathcal P_d) \to \Mod(\Sigma_g)$ into the mapping class group of the fiber. We show two results concerning the image of $\rho_d$. First, we show that the presence of an invariant known as a ``$n$-spin structure'' constrains the image in ways not predicted by previous work of Beauville \cite{beauville}. Second, we show that for $d = 5$, our invariant is the only obstruction for a mapping class to be contained in the image. This requires combining the algebro-geometric work of L\"onne \cite{lonne} with Johnson's theory \cite{johnsonfg} of the Torelli subgroup of $\Mod(\Sigma_g)$.
\end{abstract}

\section{Introduction}
Let $\mathcal{P}_d$ denote the moduli space of smooth degree-$d$ plane curves.\footnote{See Section \ref{subsection:agbasics} for a review of these algebro-geometric notions.} The tautological family of plane curves over $\mathcal{P}_d$ determines a monodromy representation
	\[
	\rho_d: \pi_1(\mathcal{P}_d) \to \Mod(\Sigma_g),
	\]
	where $g = {d-1 \choose 2}$ and $\Mod(\Sigma_g)$ is the mapping class group of the surface $\Sigma_g$ of genus $g$. This note concerns the the problem of computing the image of $\rho_d$. 
	
	The first step towards determining the image of $\rho_d$ has been carried out by A. Beauville in \cite{beauville}, building off of earlier work of W. Janssen \cite{janssen} and S. Chmutov \cite{chmutov}. Let $\Psi: \Mod(\Sigma_g) \to \Sp_{2g}(\Z)$ denote the symplectic representation on $H_1(\Sigma_g; \Z)$. Beauville has determined $\Psi \circ \rho_d$; he shows that for $d$ even it is a surjection, while for $d$ odd it is the (finite-index) stabilizer of a certain spin structure. {\em A priori}, it is therefore possible that $\rho_d$ could surject onto $\Mod(\Sigma_g)$ or a spin mapping class group, depending on the parity of $d$. 
	
	The first theorem of the present paper is that in general, this does not happen. We show that a so-called {\em $n$-spin structure} provides an obstruction for $f \in \Mod(\Sigma_g)$ to be contained in $\im(\rho_d)$, and that this obstruction is not detectable on the level of homology, i.e. that Beauville's ``upper bound'' on $\im(\rho_d)$ is not sharp. 
	\begin{maintheorem}\label{theorem:contain}
	For all $d \ge 4$, there is a finite-index subgroup $\Mod(\Sigma_g)[\phi_d] \le \Mod(\Sigma_g)$ for which
	\[
	\im(\rho_d) \subseteq \Mod(\Sigma_g)[\phi_d].
	\]
	For $d \ge 6$, the containment 
	\[
	\Mod(\Sigma_g)[\phi_d] \subsetneqq \Psi^{-1}(\Psi(\Mod(\Sigma_g)[\phi_d]))
	\]
	is {\em strict}. Consequently, for $d \ge 6$, the same is true for $\im(\rho_d)$:
	\[
	\im(\rho_d) \subsetneqq \Psi^{-1}(\Psi(\im (\rho_d))).
	\]
	\end{maintheorem}

In the statement of the theorem, $\phi_d$ is a cohomology class in $H^1(T^{*,1}\Sigma_g; \Z / (d-3) \Z)$, where $T^{*,1}\Sigma_g$ denotes the unit cotangent bundle of $\Sigma_g$, and $\Mod(\Sigma_g)[\phi_d]$ denotes the stabilizer of $\phi_d$ in the natural action of $\Mod(\Sigma_g)$ on $H^1(\Sigma_g; \Z/(d-3)\Z)$. The class $\phi_d$ is an instance of an $n$-spin structure for $n = d-3$, and is constructed in a natural way from a $(d-3)^{rd}$ root of the canonical bundle of a plane curve. Such objects, and the subgroups of $\Mod(\Sigma_g)$ fixing the set of {\em all} $n$-spin structures, were studied by P. Sipe \cite{sipe1, sipe2}. 

Theorem \ref{theorem:contain} will be proved by giving a construction of $\phi_d$ that makes the invariance of $\phi_d$ under $\im(\rho_d)$ transparent. Using a topological interpretation of $n$-spin structures based on the work of S. Humphries-D. Johnson \cite{HJ}, it will then be possible to see how the invariance of $\phi_d$ provides a strictly stronger constraint on $\im(\rho_d)$ than that of Beauville.

The second half of the paper concerns the problem of determining {\em sufficient} conditions for an element $f \in \Mod(\Sigma_g)$ to be contained in $\im(\rho_d)$.

\begin{maintheorem}\label{theorem:d5}
For $d = 5$, there is an equality
\[
\im(\rho_5) = \Mod(\Sigma_6)[\phi_5].
\]
Here $\phi_5 \in H^1(T^{*,1}\Sigma_6; \Z/2\Z)$ is a (classical) spin structure of odd parity, and $\Mod(\Sigma_6)[\phi_5]$ denotes its stabilizer within $\Mod(\Sigma_6)$. 
\end{maintheorem}
Analogous theorems hold for $d = 3, 4$ as well. The case $d = 3$ (where $g=1$) follows immediately from Beauville's computation, in light of the fact that $\Psi$ is an isomorphism $\Psi: \Mod(\Sigma_1) \to SL_2(\Z)$ for $g=1$. This case is also closely related to the work of I. Dolgachev - A. Libgober \cite{DL}. The case $d = 4$ (asserting the surjectivity $\im(\rho_4) = \Mod(\Sigma_3)$) was established by Y. Kuno \cite{kuno}. Kuno's methods are very different from those of the present paper, and make essential use of the fact that the generic curve of genus $3$ is a plane curve of degree $4$. Theorem \ref{theorem:d5} thus treats the first case where planarity is an exceptional property for a curve to possess, and shows that despite this, the monodromy of the family of plane curves of degree $5$ is still very large.

Theorem \ref{theorem:d5} is obtained by a novel combination of techniques from algebraic geometry and the theory of the mapping class group. The starting point is Beauville's work, which allows one to restrict attention to $\im(\rho_5) \cap \mathcal I_6$, where $\mathcal I_6$ is the {\em Torelli group}.\footnote{See Section \ref{subsection:johnson} for the definition of the Torelli group.}

The bridge between algebraic geometry and mapping class groups arises from the work of M. L\"onne \cite{lonne}. The main theorem of \cite{lonne} gives an explicit presentation for the fundamental group of the space $\mathcal P_{n,d}$ of smooth hypersurfaces in $\C P^n$ of degree $d$. Picard-Lefschetz theory allows one to recognize L\"onne's generators as Dehn twists. Theorem \ref{theorem:d5} is then proved by carrying out a careful examination of the configuration of vanishing cycles as simple closed curves on a surface of genus $6$. This analysis is used to exhibit the elements of Johnson's generating set for the Torelli group inside $\im(\rho_5)$. 

In genus $6$, Johnson's generating set has $4470$ elements. In order to make this computation tractable, we find a new relation in $\Mod(\Sigma_g)$ known as the ``genus-$g$ star relation''. Using this, we reduce the problem to eight easily-verified cases. An implicit corollary of the proof is a determination of a simple finite set of Dehn twist generators for the spin mapping class group $\Mod(\Sigma_6)[\phi_5]$. An alternative set of generators was obtained by Hirose \cite[Theorem 6.1]{hirose}.\\

\para{Outline} Section \ref{section:construction} is devoted to the construction of $\phi_d$. In Section \ref{section:consequences}, we recall some work of S. Humphries and D. Johnson that relates $H^1(T^{*,1} \Sigma_g; V)$ for an abelian group $V$ to the notion of a ``generalized winding number function''. We will use this perspective to show that the invariance of $\phi_d$ under $\im(\rho_d)$ provides an obstruction to the surjectivity of $\rho_d$. 

The proof of Theorem \ref{theorem:d5} is carried out in sections \ref{section:mod} through \ref{section:proofB}. Section \ref{section:mod} collects a number of results from the theory of mapping class groups. Section \ref{section:lonne} recalls L\"onne's presentation and establishes some first properties of $\im(\rho_d)$. Section \ref{section:config} continues the analysis of $\im(\rho_d)$. Finally Section \ref{section:proofB} collects these results together to prove Theorem \ref{theorem:d5}.

\para{Acknowledgements} The author would like to thank Dan Margalit for a series of valuable discussions concerning this work. He would also like to thank Benson Farb for alerting him to L\"onne's work and for extensive comments on drafts of this paper, as well as ongoing support in his role as advisor.

\section{$n^{th}$ roots of the canonical bundle and generalized spin structures}\label{section:construction}
\subsection{Plane curves and $\mathcal{P}_d$}\label{subsection:agbasics} By definition, a (projective) {\em plane curve} of degree $d$ is the vanishing locus $V(f)$ in $\C P^2$ of a nonzero homogeneous polynomial $f(x,y,z)$ of degree $d$. The space of all plane curves is identified with $\C P^N$, where $N = {d+2 \choose 2} - 1$. A plane curve $X$ of degree $d$ is {\em smooth} if $X \cong \Sigma_g$ with $g = {d-1 \choose 2}$, and otherwise $X$ is said to be {\em singular}.

	We define the {\em discriminant} as the set
	\[
	\mathcal D_d = \{f \in \C P^N \mid V(f) \mbox{ is singular.} \}.
	\]
	The discriminant $\mathcal D_d$ is the vanishing locus of a polynomial $p_d$ known as the {\em discriminant polynomial}, and is therefore a hypersurface in $\C P^N$. The space of smooth plane curves is then defined as
	\[
	\mathcal P_d = \C P^N \setminus \mathcal D_d. 
	\] 

The {\em universal family of plane curves} is the space $\mathfrak{X}_d \subset \mathcal{P}_d \times \C \P^2$ defined via
\[
\mathfrak{X}_d = \{(f,[x:y:z]) \in \mathcal{P}_d \times \C \P^2 \mid f(x,y,z) = 0 \}.
\]
The projection $\pi: \mathfrak{X}_d \to \mathcal{P}_d$ is the projection map for a $C^\infty$ fiber bundle structure on $\mathfrak{X}_d$ with fibers diffeomorphic to $\Sigma_g$.

\subsection{$n$-spin structures} Let $X$ be a smooth projective algebraic curve over $\C$ and let $K \in \Pic(X)$ denote the canonical bundle.\footnote{Recall that the {\em canonical bundle} is the line bundle whose underlying $\R^2$ bundle is $T^*X$, the cotangent bundle.} Recall that a {\em spin structure} on $X$ is an element $L \in \Pic(X)$ satisfying $L^{\otimes 2} = K$. This admits an obvious generalization.
\begin{definition}\label{definition:genspin}
An $n$-spin structure is a line bundle $L \in \Pic(X)$ satisfying $L^{\otimes n} = K$. 
\end{definition}

Let $T^{*,1}X$ denote the unit cotangent bundle of $X$, relative to an arbitrary Riemannian metric on $X$. Just as ordinary spin structures are closely related to $H^1(T^{*,1}X; \Z/2\Z)$, there is an analogous picture of $n$-spin structures.

\begin{proposition}\label{proposition:nspincover}
Let $L$ be an $n$-spin structure on $X$. Associated to $L$ are
\begin{enumerate}
\item a regular $n$-sheeted covering space $\widetilde{T^{*,1}X} \to T^{*,1}X$ with deck group $\Z/n\Z$, and
\item\label{item:primitive} a cohomology class $\phi_L \in H^1(T^{*,1}X; \Z/n\Z)$ restricting to a generator of the cohomology $H^1(S^1; \Z/n\Z)$ of the fiber of $T^{*,1}X \to X$. 
\end{enumerate}
\end{proposition}

\begin{proof}
In view of the equality $L^{\otimes n} = K$ in $\Pic(X)$, taking $n^{th}$ powers in the fiber induces a map $\mu: L \to K$. Let $L^\circ$ denote the complement of the zero section in $L$, and define $K^\circ$ similarly. Then $\mu: L^\circ \to K^\circ$ is an $n$-sheeted covering space with deck group $\Z/n\Z$ induced from the multiplicative action of the $n^{th}$ roots of unity. The covering space $\widetilde{T^{*,1}X} \to T^{*,1}X$ is obtained from $L^\circ \to K^\circ$ by restriction.

As $\widetilde{T^{*,1}X} \to T^{*,1}X$ is a regular cover with deck group $\Z/n\Z$, the Galois correspondence for covering spaces asserts that $\widetilde{T^{*,1}X}$ is associated to some homomorphism $\phi_L: \pi_1(T^{*,1}X) \to \Z/n\Z$. This gives rise to a class, also denoted $\phi_L$, in $H^1(T^{1,*}X; \Z/n\Z)$. On a given fiber of $T^{*,1}X \to X$, the covering $\widetilde{T^{*,1}X}\to T^{*,1}X$ restricts to an $n$-sheeted cover $S^1 \to S^1$; this proves the assertion concerning the restriction of $\phi_L$ to $H^1(S^1; \Z/n\Z)$. 
\end{proof}

Our interest in $n$-spin structures arises from the fact that degree-$d$ plane curves are equipped with a canonical $(d-3)$-spin structure.
\begin{fact}\label{fact:spin}
Let $X$ be a smooth degree-$d$ plane curve, $d \ge 3$. The canonical bundle $K \in \Pic(X)$ is induced from the restriction of $\mathcal{O}(d-3) \in \Pic(\C\P^2)$. Consequently, $\mathcal{O}(1)$ determines a $(d-3)$-spin structure on $X$ for $d \ge 4$. 
\end{fact}

Let $\varpi: \mathfrak{X}_d \to \C \P^2$ denote the projection onto the second factor. Then $\varpi^*(\mathcal{O}(d-3)) \in \Pic(\mathfrak{X}_d)$ restricts to the canonical bundle on each fiber, and $\varpi^*(\mathcal{O}(1))$ determines a $(d-3)$-spin structure. Let $T^{*,1}\mathfrak{X}_d$ denote the $S^1$-bundle over $\mathfrak{X}_d$ for which the fiber over $x \in X$ consists of the unit cotangent vectors $T^{*,1}_x X$. 

\begin{definition}\label{definition:phi}
The cohomology class 
\[
\phi_d \in H^1(T^{*,1}\mathfrak{X}_d; \Z/(d-3)\Z)
\]
is obtained by applying the construction of Proposition \ref{proposition:nspincover} to the pair of line bundles $\varpi^*(\mathcal{O}(1))$, \newline $\varpi^*(\mathcal{O}(d-3)) \in \Pic(\mathfrak{X}_d)$. 
\end{definition}

\section{Generalized winding numbers and obstructions to surjectivity}\label{section:consequences}
In this section, we show that the existence of $\phi_d$ gives rise to an obstruction for a mapping class $f \in \Mod(\Sigma_g)$ to be contained in $\im(\rho_d)$. For any system of coefficients $V$, there is a natural action of $\Mod(\Sigma_g)$ on $H^1(T^{*,1}\Sigma_g; V)$ which extends the action of $\Mod(\Sigma_g)$ on $H^1(\Sigma_g; V)$ via $\Psi$. To prove Theorem \ref{theorem:contain}, it therefore suffices to show that the stabilizer $\Mod(\Sigma_g)[\phi_d]$ of each nonzero element of $H^1(T^{*,1}\Sigma_g; \Z/(d-3)\Z)$ is not the full group $\Psi^{-1}(\Psi(\im(\rho_d)))$. 

The natural setting for what follows is in the unit {\em tangent} bundle of a surface, which we write $T^1\Sigma$. Of course, a choice of Riemannian metric on $\Sigma$ identifies $T^1 \Sigma$ with $T^{*,1} \Sigma$, and a choice of metric in each fiber identifies $T^{*,1} \mathfrak{X}_d$ with the ``vertical unit tangent bundle'' $T^1 \mathfrak{X}_d$; we will make no further comment on these matters.

The basis for our approach is the work of Humphries-Johnson \cite{HJ}, which gives an interpretation of $H^1(T^1 \Sigma_g; V)$ as the space of ``$V$-valued generalized winding number functions''. A basic notion here is that of a {\em Johnson lift}. For our purposes, a {\em simple closed curve} is a $C^1$-embedded $S^1$-submanifold.

\begin{definition}\label{definition:johnsonlift}
Let $a$ be a simple closed curve on the surface $\Sigma$ given by a unit-speed $C^1$ embedding $a: S^1 \to \Sigma$. A choice of orientation on $S^1$ induces an orientation on $a$, as well as providing a coherent identification $T^1_xS^1 = \{-1,1\}$ for each $x \in S^1$. The {\em Johnson lift} of $a$, written $\vec a$, is the map $\vec a: S^1 \to T^1 \Sigma$ given by 
\[
\vec a (t) = (a(t), D_ta(1)).
\]
That is, the Johnson lift of $a$ is simply the curve in $T^1\Sigma$ induced from $a$ by tracking the tangent vector. 
\end{definition}
The Johnson lift allows for the evaluation of elements of $H^1(T^1\Sigma;V)$ on simple closed curves. Let $\Sigma$ be a surface, $V$ an abelian group, and $\alpha \in H^1(T^1 \Sigma; V)$ a cohomology class. Let $a$ be a simple closed curve. By an abuse of notation, we write $\alpha(a)$ for the evaluation of $\alpha$ on the $1$-cycle determined by the Johnson lift $\vec a$. In this context we call $\alpha$ a ``generalized winding number function''.\footnote{The terminology ``generalized winding number'' is inspired by the fact that the twist-linearity property was first encountered in the context of computing winding numbers of curves on surfaces relative to a vector field.}  In \cite{HJ}, it is shown that this pairing satisfies the following properties:
\begin{theorem}[Humphries-Johnson]\label{theorem:hj}
\ 
\begin{enumerate}[(i)]
\item The evaluation $\alpha(a) \in V$ is well-defined on the isotopy class of $a$. 
\item (Twist-linearity) If $b$ is another simple closed curve and $T_b$ denotes the Dehn twist about $b$, then $\alpha$ is ``twist-linear'' in the following sense:
\begin{equation}\label{equation:TL}
\alpha(T_b(a)) = \alpha(a) + \pair{a,b} \alpha(b),
\end{equation}
where $\pair{a,b}$ denotes the algebraic intersection pairing. 

\item Let $\zeta$ be a curve enclosing a small null-homotopic disk on $\Sigma$, and let $S \subset \Sigma$ be a subsurface with boundary components $b_1, \dots, b_k$. If each $b_i$ is oriented so that $S$ is on the left and $\zeta$ is oriented similarly, then 
\begin{equation}\label{equation:chi}
\alpha(b_1) + \dots +\alpha(b_k) = \chi(S) \alpha(\zeta),
\end{equation} 
where $\chi(S)$ is the Euler characteristic of $S$. 
\end{enumerate}
\end{theorem}
\begin{remark}
Humprhies-Johnson in fact establish much more: they show that {\em every} $V$-valued twist-linear function arises as a class $\alpha \in H^1(T^1\Sigma; V)$. For what follows we only need the results of Theorem \ref{theorem:hj}. 
\end{remark}

\begin{proof}[Proof of Theorem \ref{theorem:contain}]
Consider the class $\phi_d \in H^1(T^{*,1}\mathfrak{X}_d; \Z/(d-3)\Z)$. The above discussion implies that on a given fiber $X$ of $\mathfrak{X}_d \to \mathcal{P}_d$, the restriction of $\phi_d$ determines a generalized winding number function; we write $\alpha_d \in H^1(T^1X; \Z/(d-3)\Z)$ for this class. Since $\alpha_d$ is induced from the globally-defined form $\phi_d$, it follows that $\alpha_d$ is monodromy-invariant: if $f \in \im(\rho_d)$, then for any simple closed curve $a$ on $X$, the equation
\begin{equation}\label{equation:invariant}
\alpha_d(f(a)) = \alpha_d(a)
\end{equation}
must hold. Consequently,
\[
\im(\rho_d) \subseteq \Mod(\Sigma_g)[\phi_d]
\]
as claimed.

We wish to exhibit a nonseparating simple closed curve $b$ for which $\alpha_d(b) \ne 0$. Given such a $b$, there is another simple closed curve $a$ satisfying $\pair{a,b} = 1$. Then the twist-linearity condition (\ref{equation:TL}) will show that
\[
\alpha_d(T_b(a)) = \alpha_d(a) + \alpha_d(b) \ne \alpha_d(a);
\]
this contradicts (\ref{equation:invariant}). It follows that the Dehn twist $T_b$ for such a curve cannot be contained in $\Mod(\Sigma_g)[\phi_d]$. 

In the case when $d$ is even, when $\Psi^{-1}(\Psi(\im(\rho_d))) = \Mod(\Sigma_g)$, this will prove Theorem \ref{theorem:contain}. For $d$ odd, there is an additional complication. Here, the class $\frac{d-3}{2} \phi_d \in H^1(T^{*,1}\mathfrak{X}_d; \Z/2\Z)$ determines an ordinary spin structure, and according to Beauville, the group $\Psi(\Mod(\Sigma_g)[\phi_d])$ is the stabilizer of $\frac{d-3}{2}\phi_d$ in $\Sp(2g,\Z)$. We must therefore exhibit a curve $b$ for which $\alpha_d(b)$ is nonzero and $\frac{d-3}{2}$-torsion. Equation (\ref{equation:TL}) shows that such a curve {\em does} stabilize the spin structure $\frac{d-3}{2} \phi_d$, but not the refinement to a $(d-3)$-spin structure $\phi_d$. 

It remains to exhibit a suitable curve $b$. It follows easily from the twist-linearity condition (\ref{equation:TL}) that given any subsurface $S \subset X$ of genus $1$ with one boundary component, there is some (necessarily nonseparating) curve $c$ contained in $S$ with $\alpha_d(c) = 0$. Let $S_1,S_2,S_3$ be a collection of mutually-disjoint subsurfaces of genus $1$ with one boundary component, and let $c_1, c_2, c_3$ be curves satisfying $\alpha_d(c_i) = 0$, and for which $c_i$ is contained in $S_i$ (recall that $d \ge 6$ and so the genus of $X$ is $g \ge 10$). Choose $b$ disjoint from all $c_i$ so that the collection of curves $b, c_1,c_2, c_3$ encloses a subsurface $\Sigma$ homeomorphic to a sphere with $4$ boundary components. From (\ref{equation:chi}) and the construction of the $c_i$, it follows that when $b$ is suitably oriented, it satisfies 
\[
\alpha_d(b) = \chi(\Sigma)\alpha_d(\zeta) = -2\alpha_d(\zeta).
\]
Recall that by Proposition \ref{proposition:nspincover}.\ref{item:primitive}, the element $\alpha_d(\zeta) \in \Z/(d-3)\Z$ is primitive. Thus $\alpha_d(b) \ne 0$ for any $d$, but is $\frac{d-3}{2}$-torsion when $d$ is odd, as required.
\end{proof}

\section{Results from the theory of the mapping class group}\label{section:mod}
	We turn now to the proof of Theorem \ref{theorem:d5}. From this section onwards, we adopt the conventions and notations of the reference \cite{FM}. In particular, the {\em left-handed} Dehn twist about a curve $c$ is written $T_c$, and the geometric intersection number between curves $a,b$ is written $i(a,b)$. We pause briefly to establish some further conventions. We will often refer to a simple closed curve as simply a ``curve'', and will often confuse the distinction between a curve and its isotopy class. Unless otherwise specified, we will assume that all intersections between curves are essential. 
	\subsection{The change-of-coordinates principle}\label{subsection:ccp}
	The {\em change-of-coordinates principle} roughly asserts that if two configurations of simple closed curves and arcs on a surface have the same intersection pattern, then there is a homeomorphism taking one configuration to the other. There are many variants of the change-of-coordinates principle, all based on the classification of surfaces. See the discussion in \cite[Section 1.3.2]{FM}.
	
	\para{Basic principle} Suppose $c_1, \dots, c_n$ and $d_1, \dots, d_n$ are configurations of curves on a surface $S$ all meeting transversely. The surface $\overline{S \setminus \{c_i\}}$ has a labeling on segments of its boundary, corresponding to the segments of the curves $c_i$ from which the boundary component arises. Suppose there is a homeomorphism 
	\[
	f: \overline{S \setminus \{c_i\}} \to \overline{S \setminus \{d_i\}}
	\]
	taking every boundary segment labeled by $c_i$ to the corresponding $d_i$ segment. Then $f$ can be extended to a homeomorphsim $f: S \to S$ taking the configuration $c_i$ to $d_i$.\\

	We illustrate this in the case of {\em chains}. 	
	\begin{definition}\label{definition:chain}
	Let $S$ be a surface. A {\em chain} on $S$ of length $k$ is a collection of curves $(c_1, \dots, c_k)$  for which the geometric intersection number $i(c_i, c_j)$ is $1$ if $\abs{i-j} = 1$ and $0$ otherwise. If $C = (c_1, \dots, c_k)$ is a chain, the {\em boundary} of $C$, written $\partial C$, is defined to be the boundary of a small regular neighborhood of $c_1 \cup \dots \cup c_k$. When $k$ is even, $\partial C$ is a single (necessarily separating) curve, and when $k$ is odd, $\partial C = d_1 \cup d_2$ consists of two curves $d_1, d_2$ whose union separates $S$.
	\end{definition}
	
	\begin{lemma}[Change-of-coordinates for chains]\label{lemma:chainCCP}
	Let $(c_1, \dots, c_{k})$ and $(d_1, \dots, d_{k})$ be chains of even length $k$ on a surface $S$. Then there is a homeomorphism $f: S \to S$ for which $f(c_i) = d_i, 1\le i \le k$. 
	\end{lemma}
	\begin{proof}
	See \cite[Section 1.3.2]{FM}.
	\end{proof}

	\subsection{Some relations in the mapping class group}
	\begin{proposition}[Braid relation]\label{proposition:braid}
	Let $S$ be a surface, and $a,b$ curves on $S$ satisfying $i(a,b) = 1$. Then
	\begin{equation}\label{equation:braid}
	T_a T_b T_a = T_b T_a T_b.
	\end{equation}
	On the level of curves,
	\[
	T_a T_b (a) = b.
	\]
	Any such $a,b$ are necessarily non-separating.
	
	Conversely, if $a,b$ are curves on $S$ in distinct isotopy classes that satisfy the braid relation (\ref{equation:braid}), then $i(a,b) = 1$. 
	\end{proposition}	
	\begin{proof}
	See \cite[Proposition 3.11]{FM} for the proof of the first assertion, and \cite[Proposition 3.13]{FM} for the second.
	\end{proof}

	\para{The chain relation}   The {\em chain relation} relates Dehn twists about curves in a chain to Dehn twists around the boundary. We will require a slightly less well-known form of the chain relation for chains of odd length; see \cite[Section 4.4.1]{FM} for details.
	\begin{proposition}[Chain relation]\label{proposition:chain}
	Let $C = (c_1, \dots, c_k)$ be a chain with $k$ odd. Let $d_1, d_2$ denote the components of $\partial C$. Then the following relation holds:
	\[
	(T_{c_1}^2 T_{c_2} \dots T_{c_k})^k = T_{d_1} T_{d_2}.
	\]
	\end{proposition}
	
	\para{The genus-$g$ star relation} We will also need to make use of a novel relation generalizing the star relation (setting $g=1$ below recovers the classical star relation). 
	\begin{proposition}[Genus-$g$ star relation]\label{proposition:star}
	With reference to the curves $a_1, a_2, c_1, \dots, c_{2g}, d_1, d_2, d_3$ on the surface $\Sigma_{g,3}$ of Figure \ref{figure:ramify}, the following relation holds in $\Mod(\Sigma_{g,3})$:
	\begin{equation}\label{equation:star}
	(T_{a_1} T_{a_2} T_{c_1} \dots T_{c_{2g}})^{2g+1} = T_{d_1}^g T_{d_2} T_{d_3}.
	\end{equation}
	\end{proposition}
	\begin{proof}
	\begin{figure}
	\labellist
	\small
	\pinlabel $d_1$ at 25 100
	\pinlabel $d_2$ at 212 106
	\pinlabel $d_3$ at 212 40
	\pinlabel $a_1$ at 95 100
	\pinlabel $a_2$ at 95 40
	\pinlabel $c_1$ at 60 89
	\pinlabel $c_2$ at 100 85
	\pinlabel $c_{2g-1}$ at 160 95
	\pinlabel $c_{2g}$ at 183 85
	\pinlabel $p_1$ at 327 125
	\pinlabel $p_{2g+1}$ at 278 100
	\pinlabel $p_2$ at 364 106
	\pinlabel $\alpha$ at 314 100
	\pinlabel $\sigma_2$ at 384 80
	\pinlabel $\delta_1$ at 327 72
	\pinlabel $\delta_2$ at 300 146
	\endlabellist
	\includegraphics{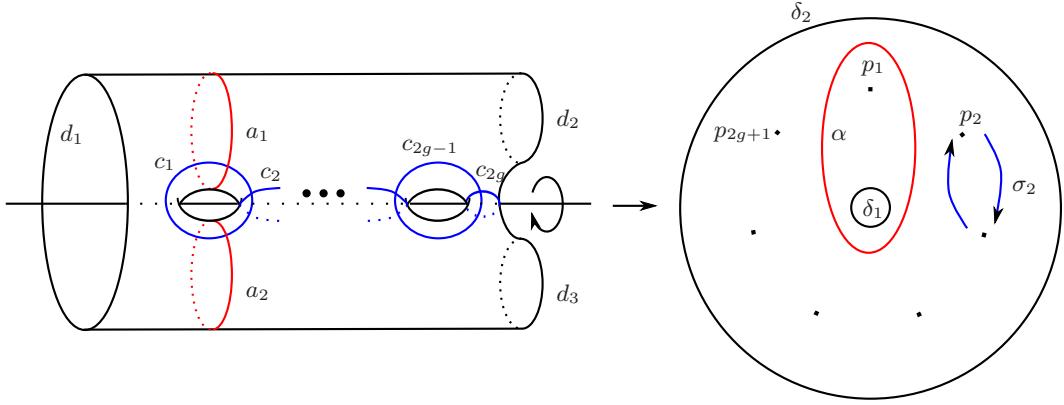}
	\caption{The genus-$g$ star relation.}
	\label{figure:ramify}
	\end{figure}
	
	We will derive the genus-$g$ star relation from a more transparent relation in a braid group. Figure \ref{figure:ramify} depicts a $2:1$ covering $\Sigma_{g,3} \to \Sigma_{0,2}$ ramified at $2g+1$ points. Number the ramification points clockwise $p_1, \dots, p_{2g+1}$, and consider the mapping class group $\operatorname{Mod}(\Sigma_{0,2,2g+1})$ relative to these points. Under the covering, the double-twist $T_{\delta_1}^2$ lifts to $T_{d_1} \in \Mod(\Sigma_{g,3})$, and the twist $T_{\delta_2}$ lifts to $T_{d_2} T_{d_3}$. The twist $T_\alpha$ lifts to $T_{a_1} T_{a_2}$, and the half-twist $\sigma_i$ lifts to $T_{c_i}$. 
	Let $f \in \operatorname{Mod}(\Sigma_{0,2,2g+1})$ be the push map moving each $p_i$ clockwise to $p_{i+1}$, with subscripts interpreted mod $2g+1$. One verifies the equality
	\[
	f = T_\alpha \sigma_1 \dots \sigma_{2g} T_{\delta_1}^{-1}.
	\]
	It follows that
	\[
	f^{2g+1} = (T_\alpha \sigma_1 \dots \sigma_{2g})^{2g+1} T_{\delta_1}^{-(2g+1)},
	\]
	since $T_{\delta_1}$ is central. As $f^{2g+1}$ is the push map around the core of the annulus, there is an equality
	\[
	f^{2g+1} = T_{\delta_1}^{-1} T_{\delta_2}.
	\]
	Combining these results,
	\begin{equation}\label{equation:flap}
	T_{\delta_2} T_{\delta_1}^{2g} = (T_\alpha \sigma_1 \dots \sigma_{2g})^{2g+1}.
	\end{equation}
	Under the lifting described above, the relation (\ref{equation:flap}) in $\operatorname{Mod}(\Sigma_{0,2,2g+1})$ lifts to the relation (\ref{equation:star}) in $\Mod(\Sigma_{g,3})$.
	\end{proof}
	
	\subsection{The Johnson generating set for $\mathcal{I}_g$}\label{subsection:johnson}
	There is a natural map
	\[
	\Psi: \Mod(\Sigma_g) \to \Sp_{2g}(\Z)
	\]
	taking a mapping class $f$ to the induced automorphism $f_*$ of $H_1(\Sigma_g; \Z)$. The {\em Torelli group} $\mathcal{I}_g$ is defined to be the kernel of this map:
	\[
	\mathcal{I}_g = \ker(\Psi).
	\]
	In \cite{johnsonfg}, Johnson produced a finite set of generators for $\mathcal{I}_g$, for all $g \ge 3$. Elements of this generating set are known as {\em chain maps}. Let $C = (c_1, \dots, c_k)$ be a chain of {\em odd} length with boundary $\partial C = d_1 \cup d_2$. There are exactly two ways to orient the collection of curves $c_1, \dots, c_k$ in such a way that the algebraic intersection number $c_i \cdot c_{i+1} = +1$. Relative to such a choice, the {\em chain map} associated to $C$ is then the mapping class $T_{d_1} T_{d_2}^{-1}$, where $d_1$ is distinguished as the boundary component to the {\em left} of the curves $c_1, c_3, \dots, c_k$. The mapping class $T_{d_1} T_{d-2}^{-1}$ is also called the {\em bounding pair map} for $d_1, d_2$. 
	
	While a complete description of Johnson's generating set is quite tidy and elegant, it has the disadvantage of requiring several preliminary notions before it can be stated. We therefore content ourselves with a distillation of his work that is more immediately applicable to our situation. 
	
	\begin{figure}[ht!]
	\labellist
	\small
	\pinlabel $c_1$ at 15 68
	\pinlabel $c_2$ at 45 80
	\pinlabel $c_3$ at 75 68
	\pinlabel $c_4$ at 105 80
	\pinlabel $c_5$ at 135 68
	\pinlabel $c_6$ at 160 80
	\pinlabel $c_{2g-2}$ at 240 80
	\pinlabel $c_{2g-1}$ at 270 68
	\pinlabel $c_{2g}$ at 300 80
	\pinlabel $\beta$ at 130 90
	\endlabellist
	\includegraphics{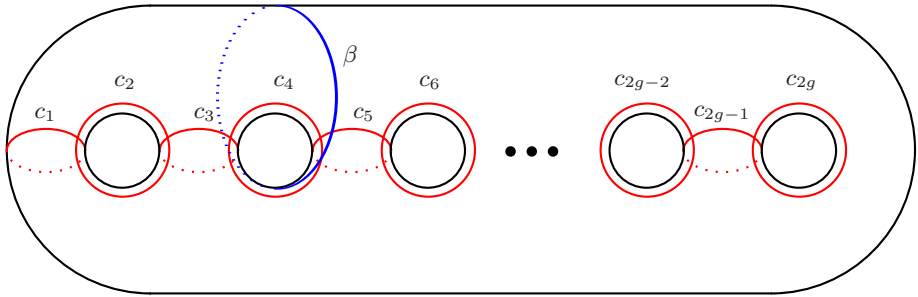}
	\caption{Curves involved in the Johnson generating set.}
	\label{figure:johnson}
	\end{figure}

	\begin{theorem}[Johnson]\label{theorem:genset}
	For $g \ge 3$, let $\Gamma \le \Mod(\Sigma_g)$ be a subgroup that contains the Dehn twists about the curves $c_1, \dots, c_{2g}$ shown in Figure \ref{figure:johnson}. Suppose that $\Gamma$ contains all chain maps for the odd-length chains of the form $(c_1, \dots, c_k)$ and $(\beta, c_5, \dots, c_k)$. Then $\mathcal{I}_g \le \Gamma$. 
	\end{theorem}
	\begin{proof}
	The interested reader should have no trouble deducing Theorem \ref{theorem:genset} from the Main Theorem and Lemma 1(f) of \cite{johnsonfg}.  
	\end{proof}
	
\section{The L\"onne presentation} \label{section:lonne}
	In this section, we recall L\"onne's work \cite{lonne} computing $\pi_1(\mathcal P_d)$, and apply this to derive some first properties of the monodromy map $\rho_d: \pi_1(\mathcal P_d) \to \Mod(\Sigma_g)$. 
	\subsection{Picard-Lefschetz theory}\label{subsection:PL} Picard-Lefschetz theory concerns the problem of computing monodromies attached to singular points of holomorphic functions $f: \C^n \to \C$. This then serves as the local theory underpinning more global monodromy computations. Our reference is \cite{arnold}. 
	
	Let $U \subset \C^2$ and $V \subset \C$ be open sets for which $0 \in V$. Let $f(u,v): U \to V$ be a holomorphic function. Suppose $f$ has an isolated critical value at $z = 0$, and that there is a single corresponding critical point $p \in \C^2$. Suppose that $p$ is {\em of Morse type} in the sense that the Hessian
	\[
\abcd{\frac{\partial^2 f}{\partial^2 x}}{\frac{\partial^2 f}{\partial x \partial y}}{\frac{\partial^2 f}{\partial y \partial x}}{\frac{\partial^2 f}{\partial^2 y}}
	\]
	is non-singular at $p$. 
	
	In such a situation, the fiber $f^{-1}(z)$ for $z \ne 0$ is diffeomorphic to an open annulus. The core curve of such an annulus is called a {\em vanishing cycle}. Let $\gamma$ be a small circle in $\C$ enclosing only the critical value at $z = 0$. Let $z_1 \in \gamma$ be a basepoint with corresponding core curve $c\subset f^{-1}(z_1)$. The Picard-Lefschetz theorem describes the monodromy obtained by traversing $\gamma$.

	\begin{theorem}[Picard-Lefschetz for $n = 2$]\label{theorem:PL}
	 With reference to the preceding discussion, the monodromy $\mu \in \Mod(f^{-1}(z_1))$ attached to traversing $\gamma$ counter-clockwise is given by a right-handed Dehn twist about the vanishing cycle:
	\[
	\mu = T_c^{-1}.
	\]
	\end{theorem}	

More generally, let $D^*$ denote the punctured unit disk
\[
D^* = \{w \in \C \mid 0 < \abs w \le 1\},
\]
and write $D = \{w \in \C \mid \abs{w} \le 1\}$ for the closed unit disk.

Let $f(x,y,z)$ be a homogeneous polynomial of degree $d$ with the following properties: 
\begin{enumerate}
\item For $c \in D$, the plane curve $c z^d - f(x,y,z)$ is singular only for $c = 0$.
\item The only critical point for $f$ of the form $(x,y,0)$ is the point $(0,0,0)$.  
\item The function $f(x,y,1)$ has a single critical point of Morse type at $(x,y) = (0,0)$.
\end{enumerate}
In this setting, the local theory of Theorem \ref{theorem:PL} can be used to analyze the monodromy of the family 
\[
E \subset D^* \times \C P^2 = \{(c, [x:y:z]) \mid cz^d = f(x,y,z) \}
\]
around the boundary $\partial D^*$.

\begin{theorem}[Picard-Lefschetz for plane curve families]\label{theorem:PLglobal}
Let $f \in \C P^N$ satisfy the properties (1), (2), (3) listed above. Let $X = V(z^d - f(x,y,z))$ denote the fiber above $1 \in D^*$. Then there is a vanishing cycle $c \subset X$ so that the monodromy $\mu \in \Mod(X)$ obtained by traversing $\partial D^*$ counter-clockwise is given by a right-handed Dehn twist about the vanishing cycle:
\[
\mu = T_c^{-1}.
\]
\end{theorem}
	
\begin{proof}
Condition (2) asserts that the monodromy can be computed by restricting attention to the affine subfamily obtained by setting $z = 1$. Define
\[
E^\circ \subset D^* \times \C^2 = \{(c,x,y) \mid c = f(x,y,1) \}.
\]
Define $U = \{(x,y) \in \C^2 \mid \abs{f(x,y,1)} \le 1\}$, and consider $f(x,y,1)$ as a holomorphic function $f: U \to D$. The monodromy of this family then corresponds to the monodromy of the original family $E \to D^*$. The result now follows from Condition (3) in combination with Theorem \ref{theorem:PL} as applied to $f(x,y,1)$. 
\end{proof}

	\subsection{L\"onne's theorem} There are some preliminary notions to establish before L\"onne's theorem can be stated. We begin by introducing the {\em L\"onne graphs} $\Gamma_d$. L\"onne obtains his presentation of $\pi_1(\mathcal{P}_d)$ as a quotient a certain group constructed from $\Gamma_d$. 
	\begin{definition}\label{definition:lonnegraph}[L\"onne graph]
	Let $d \ge 3$ be given. The {\em L\"onne graph} $\Gamma_d$ has vertex set 
	\[
	I_d = \{(a,b) \mid 1 \le a,b \le d-1 \}.
	\]
	Vertices $(a_1,b_1)$ and $(a_2, b_2)$ are connected by an edge if and only if both of the following conditions are met:
	\begin{enumerate}
	\item $\abs{a_1-a_2} \le 1$ and $\abs{b_1-b_2} \le 1$.
	\item $(a_1-a_2)(b_1 - b_2) \le 0$.
	\end{enumerate}
	The set of edges of $\Gamma_d$ is denoted $E_d$.
	\end{definition}
	\begin{figure}[ht!]
	\begin{center}
	\includegraphics{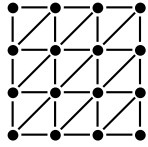}
	\caption{The L\"onne graph $\Gamma_5$.}
	\label{figure:gamma5}
	\end{center}
	\end{figure}
	
	Vertices $i,j,k \in \Gamma_d$ are said to form a {\em triangle} when $i,j,k$ are mutually adjacent. The triangles in the L\"onne graph are crucial to what follows. It will be necessary to endow them with orientations.
	\begin{definition}[Orientation of triangles]\label{definition:orient}
	Let $i,j,k$ determine a triangle in $\Gamma_d$.
	\begin{enumerate}
	\item If 
	\[
	i = (a,b),\quad j= (a,b+1), \quad k = (a+1,b),
	\]
	then the triangle $i,j,k$ is positively-oriented by traversing the boundary clockwise.
	\item If
	\[
	i = (a,b), \quad j = (a,b+1), \quad k = (a-1,b+1),
	\]
	then the triangle $i,j,k$ is positively-oriented by traversing the boundary counterclockwise. 
	\end{enumerate}
	We say that the ordered triple $(i,j,k)$ of vertices determining a triangle is positively-oriented if traversing the boundary from $i$ to $j$ to $k$ agrees with the orientation specified above.
	\end{definition}
	
	\begin{definition}[Artin group]
	Let $\Gamma$ be a graph with vertex set $V$ and edge set $E$. The {\em Artin group} $A(\Gamma)$ is defined to be the group with generators
	\[
	\sigma_i, i \in V,
	\]
	subject to the following relations:
	\begin{enumerate}
	\item $\sigma_i \sigma_j = \sigma_j \sigma_i$ for all $(i,j) \not \in E$.
	\item $\sigma_i \sigma_j \sigma_i = \sigma_j \sigma_i \sigma_j$ for all $(i,j) \in E$.
	\end{enumerate}
	\end{definition}
	
	\begin{theorem}[L\"onne]\label{theorem:lonne}
	For $d \ge 3$, the group $\pi_1(\mathcal{P}_d)$ is isomorphic to a quotient of the Artin group $A(\Gamma_d)$, 
	subject to the following additional relations:
	\begin{enumerate}
	\setcounter{enumi}{2}
	 \item \label{item:tringle}
	$\sigma_i \sigma_j \sigma_k \sigma_i = \sigma_j \sigma_k \sigma_i \sigma_j$ if $(i,j,k)$ forms a positively-oriented triangle in $\Gamma_d$.
	\item An additional family of relations $R_i, i \in I_d$.
	\item An additional relation $\widetilde R$.
	\end{enumerate}
	\end{theorem}
	
	\begin{remark} Define the group $B(\Gamma_d)$ as the quotient of the Artin group $A(\Gamma_d)$ by the family of relations (\ref{item:tringle}) in Theorem \ref{theorem:lonne}. As our statement of L\"onne's theorem indicates, the additional relations will be of no use to us, and our theorem really concerns the lift of the monodromy representation $\tilde \rho_d: B(\Gamma_d) \to \Mod(\Sigma_g)$. 
	\end{remark}
	
	For the analysis to follow, it is essential to understand the mapping classes $\rho_d(\sigma_i), i \in I_d$. 
	
	\begin{proposition}\label{proposition:genimg}
	For each generator $\sigma_i$ of Theorem \ref{theorem:lonne}, the image
	\[
	\rho_d(\sigma_i) = T_{c_i}^{-1}
	\]
	is a right-handed Dehn twist about some vanishing cycle $c_i$ on a fiber $X \in \mathcal P_d$. 
	\end{proposition}
	
	\begin{proof}
	The result will follow from Theorem \ref{theorem:PLglobal}, once certain aspects of L\"onne's proof are recalled. 
	
	The generators $\sigma_i$ of Theorem \ref{theorem:lonne} correspond to specific loops in $\mathcal P_d$ known as {\em geometric elements}. 
\begin{definition}[Geometric element]
Let $D = V(p)$ be a hypersurface in $\C^n$ defined by some polynomial $p$. An element $x \in \pi_1(\C^n \setminus D)$ that can be represented by a path isotopic to the boundary of a small disk transversal to $D$ is called a {\em geometric element}. If $\widetilde D$ is a projective hypersurface, an element $x \in \pi_1(\C P^n \setminus \widetilde D)$ is said to be a geometric element if it can be represented by a geometric element in some affine chart.
\end{definition}
	In L\"onne's terminology, the generators $\sigma_i, i \in I_d$ arise as a ``Hefez-Lazzeri basis'' - this will require some explanation. Consider the linearly-perturbed Fermat polynomial
	\[
	f(x,y,z) = x^d + y^d + \nu_x x z^{d-1}+ \nu_y y z^{d-1}
	\]
	for well-chosen constants $\nu_x, \nu_y$. Such an $f$ satisfies the conditions (1)-(3) of Theorem \ref{theorem:PLglobal} near each critical point. Moreover, there is a bijection between the critical points of $f(x,y,1)$ and the set $I_d$ of Definition \ref{definition:lonnegraph}. If $\nu_x, \nu_y$ are chosen carefully, each critical point lies above a distinct critical value - in this way we embed $I_d \subset \C$. 
	
	Each $c \in \C$ determines a plane curve $V(cz^d-f)$. The values of $c$ for which $V(c z^d-f)$ is not smooth are exactly the critical values of $f(x,y,1)$. The family
	\[
	H = \{V(c z^d-f) \mid V(c  z^d-f) \mbox{ is smooth}\}
	\]
	is a subfamily of $\mathcal{P}_{d}$ defined over $\C \setminus I_d$. The Hefez-Lazzeri basis $\{\sigma_i \mid i \in I_d\}$ is a carefully-chosen set of paths in $\C \setminus I_d$ with each $\sigma_i$ encircling an individual $i \in I_d$. L\"onne shows that the inclusions of these paths into $\mathcal{P}_d$ via the family $H$ generate $\pi_1(\mathcal P_d)$. The result now follows from an application of Theorem \ref{theorem:PLglobal}.
\end{proof}
	
	\subsection{First properties of $\rho_d$} Proposition \ref{proposition:genimg} establishes the existence of a collection $c_i, i \in I_d$ of vanishing cycles on $X$. In this section, we derive some basic topological properties of this configuration arising from the fact that the Dehn twists $T_{c_i}^{-1}$ must satisfy the relations (1)-(3) of L\"onne's presentation.
	
	\begin{lemma}\label{lemma:imgprops}\ 
	\begin{enumerate}
	\item\label{item:braid} If the vertices $v_i, v_j$ are adjacent, then the curves $c_i, c_j$ satisfy $i(c_i, c_j) = 1$. 
	\item\label{item:distinct} For $d \ge 4$, the curves $c_i, i \in I_d$ are pairwise distinct, and all $c_i$ are non-separating.
	\item\label{item:disjoint} If the vertices $v_i, v_j$ in $\Gamma_d$ are non-adjacent, then the curves $c_i$ and $c_j$ are disjoint.
	\item\label{item:triangle} For $d \ge 4$, if the vertices $v_i, v_j, v_k$ form a triangle in $\Gamma_d$, then the curves $c_i, c_j, c_k$ are supported on an essential subsurface\footnote{A subsurface $S' \subset S$ is {\em essential} if every component of $\partial S'$ is not null-homotopic.} $S_{ijk} \subset X$ homeomorphic to $\Sigma_{1,2}$. Moreover, if the triangle determined by $v_i,v_j,v_k$ is positively oriented, then $i(c_i, T^{-1}_{c_j}(c_k)) = 0$.
	\end{enumerate}
	\end{lemma}
	
	\begin{proof}
	(\ref{item:braid}): If $v_i$ and $v_j$ are adjacent, then the Dehn twists $T_{c_i}^{-1}$ and $T_{c_j}^{-1}$ satisfy the braid relation. It follows from Proposition \ref{proposition:braid} that $i(c_i,c_j) = 1$. \\
	
	\noindent (\ref{item:distinct}): Suppose $v_i$ and $v_j$ are distinct vertices. For $d \ge 4$, no two vertices have the same set of adjacent vertices, so that there is some $v_k$ adjacent to $v_i$ and not $v_j$. By (\ref{item:braid}) above, it follows that $T_{c_i}^{-1}$ and $T_{c_k}^{-1}$ satisfy the braid relation, while $T_{c_j}^{-1}$ and $T_{c_k}^{-1}$ do not, showing that the isotopy classes of $c_i$ and $c_j$ are distinct. Since each $c_i$ satisfies a braid relation with some other $c_j$, Proposition \ref{proposition:braid} shows that $c_i$ is non-separating.\\
	
	\noindent(\ref{item:disjoint}): If $v_i$ and $v_j$ are non-adjacent, then the Dehn twists $T_{c_i}^{-1}$ and $T_{c_j}^{-1}$ commute. According to \cite[Section 3.5.2]{FM}, this implies that either $c_i = c_j$ or else $c_i$ and $c_j$ are disjoint. By (\ref{item:distinct}), the former possibility cannot hold.\\
	
	\noindent(\ref{item:triangle}): Via the change-of-coordinates principle, it can be checked that if $c_i, c_j, c_k$ are curves that pairwise intersect once, then $c_i \cup c_j \cup c_k$ is supported on an essential subsurface of the form $\Sigma_{1,b}$ for $1 \le b \le 3$. In the case $b = 1$, the curve $c_k$ must be of the form $c_k = T_{c_i}^{\pm 1}(c_j)$. It follows that if $d$ is a curve such that $i(d,c_k) \ne 0$, then at least one of $i(d, c_i)$ and $i(d,c_j)$ must also be nonzero. However, as $d \ge 4$, there is always some vertex $v_l$ adjacent to exactly one of $c_i, c_j, c_k$. The corresponding curve $c_l$ would violate the condition required of $d$ above (possibly after permuting the indices $i,j,k$). 
	
	It remains to eliminate the possibility $b = 3$. In this case, the change-of-coordinates principle implies that up to homeomorphism, the curves $c_i, c_j, c_k$ must be arranged as in Figure \ref{figure:hyperbraid}. It can be checked directly (e.g. by examining the action on $H_1(\Sigma_{1,3})$) that for this configuration, the relation 
	\[
	T_{c_i}^{-1}T_{c_j}^{-1} T_{c_k}^{-1} T_{c_i}^{-1} = T_{c_j}^{-1}T_{c_k}^{-1} T_{c_i}^{-1} T_{c_j} ^{-1}
	\]
	{\em does not} hold. This violates relation (3) in L\"onne's presentation of $\pi_1(\mathcal P_d)$. We conclude that necessarily $b = 2$. 
	
	Having shown that $b = 2$, it remains to show the condition $i(c_i, T^{-1}_{c_j}(c_k)) = 0$ for a positively-oriented triangle. Let $(x,y,z)$ denote a $3$-chain on $\Sigma_{1,2}$. The change-of-coordinates principle implies that without loss of generality, $c_i=x,  c_j=y$, and $c_k = T_y^{\pm 1}(z)$. We wish to show that necessarily $c_k = T_y(z)$. It can be checked directly (e.g. by considering the action on $H_1(\Sigma_{1,2})$) that only in the case $c_k = T_y(z)$ does relation (3) in the L\"onne presentation hold. 	
	\begin{figure}[ht!]
	\begin{center}
	\includegraphics{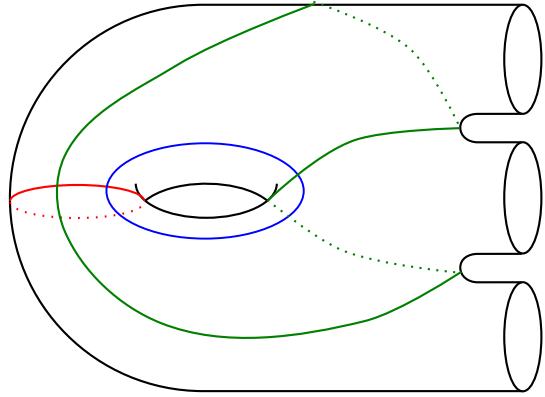}
	\caption{Lemma \ref{lemma:imgprops}.\ref{item:triangle}: the configuration of $c_i,c_j,c_k$ in the $b = 3$ case.}
	\label{figure:hyperbraid}
	\end{center}
	\end{figure}
	\end{proof}
	
\section{Configurations of vanishing cycles}\label{section:config}

The goal of this section is to derive an explicit picture of the configuration of vanishing cycles on a plane curve of degree $5$. The main result of the section is Lemma \ref{lemma:config}.

\begin{lemma}\label{lemma:config}
There is a homeomorphism $f: X \to \Sigma_6$ such that with reference to Figure \ref{figure:mainconfig},
\begin{enumerate}
\item\label{item:VC}The curves $c_1, \dots, c_{12}$ are vanishing cycles; that is, $T_{c_i} \in \im(\rho_5)$ for $1 \le i \le 12$. The curves $x,y,z$ are also vanishing cycles.
\item\label{item:square}The curve $b$ satisfies $T_b^2 \in \im(\rho_5)$.
\end{enumerate}
\end{lemma}

\begin{figure}
\labellist
\small
\pinlabel $c_1$ at 15 118
\pinlabel $c_2$ at 45 135
\pinlabel $c_3$ at 70 135
\pinlabel $c_4$ at 88 150
\pinlabel $c_5$ at 145 170
\pinlabel $c_6$ at 168 185
\pinlabel $c_7$ at 210 137
\pinlabel $c_8$ at 225 105
\pinlabel $c_9$ at 210 73
\pinlabel $c_{10}$ at 227 40
\pinlabel $c_{11}$ at 152 58
\pinlabel $c_{12}$ at 83 70
\pinlabel $c_{13}$ at 80 90

\pinlabel $x$ at 212 200
\pinlabel $y$ at 205 10
\pinlabel $z$ at 145 105
\pinlabel $b$ at 118 175
\endlabellist
\includegraphics{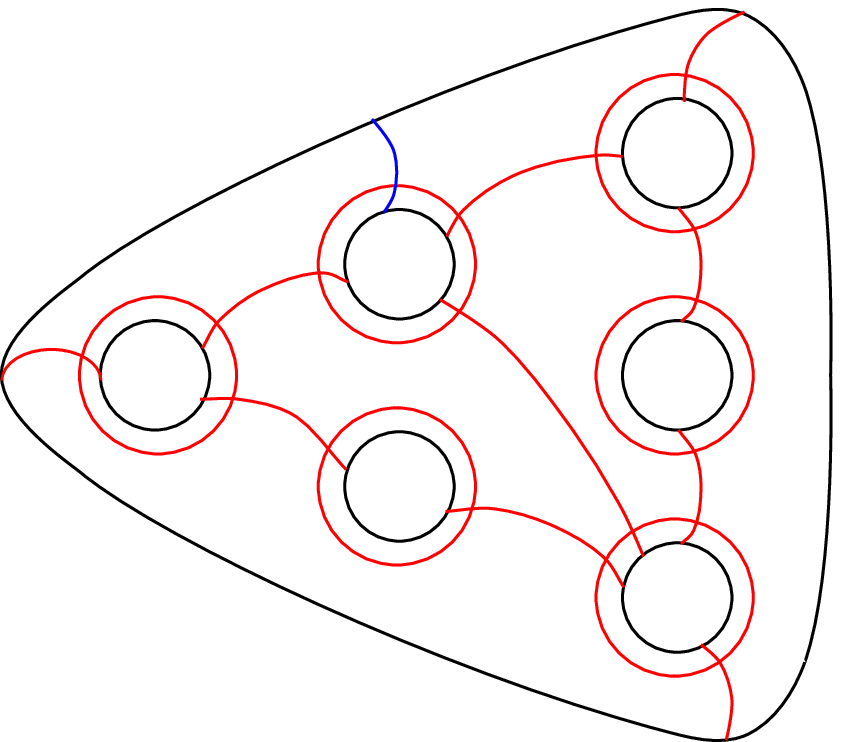}
\caption{The curves of Lemma \ref{lemma:config}. The bottom halves of curves $b,x,y,z,$ and $c_i$ for $i$ odd have been omitted for clarity; on the bottom half, each curve follows its mirror image on the top.}
\label{figure:mainconfig}
\end{figure}

\begin{proof}
Lemma \ref{lemma:config} will be proved in three steps.

\para{Step 1: Uniqueness of L\"onne configurations}
\begin{lemma}\label{lemma:unique} Suppose $d \ge 5$ is odd. Up to homeomorphism, there is a unique configuration of curves $c_i, i \in I_d$ on $\Sigma_g$ whose intersection pattern is prescribed by $\Gamma_d$ and such that the twists $T_{c_i}^{-1}$ satisfy the relations (1),(2),(3) given by L\"onne's presentation. 
\end{lemma}
A configuration of curves $c_i, i \in I_d$ as in Lemma \ref{lemma:unique} will be referred to as a {\em L\"onne configuration}. 
\begin{proof}

Let $a_{1,1},\dots, a_{d-1, d-1}$ determine a L\"onne configuration on $\Sigma_g$. We will exhibit a homeomorphism of $\Sigma_g$ taking each $a_{i,j}$ to a corresponding $b_{i,j}$ in a ``reference'' configuration $\{b_{i,j}\}$ to be constructed in the course of the proof. This will require three steps.

\para{Step 1: A collection of disjoint chains} Each row in the L\"onne graph determines a chain of length $d-1$. {The change of coordinates principle for chains of even length} (Lemma \ref{lemma:chainCCP}) asserts that any two chains of length $d-1$ are equivalent up to homeomorphism. Considering the odd-numbered rows of $\Gamma_d$, it follows that there is a homeomorphism $f_1$ of $\Sigma_g$ that takes each $a_{2i-1,j}$ for $1 \le i \le d-1$ to a curve $b_{2i-1,j}$ in a standard picture of a chain. We denote the subsurface of $\Sigma_g$ determined by the chain $a_{2i-1,1}, \dots, a_{2i-1,d-1}$ as $A_i$, and similarly we define the subsurfaces $B_i$ of the reference configuration. Each $A_i, B_i$ is homeomorphic to $\Sigma_{(d-1)/2,1}$. 

\para{Step 2: Arcs on $A_i$} The next step is to show that up to homeomorphism, there is a unique picture of what the intersection of the remaining curves $a_{2i,j}$ with $\bigcup A_i$ looks like. Consider a curve $a_{2i,j}$. Up to isotopy, $a_{2i,j}$ intersects only the subsurfaces $A_i$ and $A_{i+1}$. We claim that $a_{2i,j}$ can be isotoped so that its intersection with $A_i$ is a single arc, and similarly for $A_{i+1}$. If $j = d-1$, then $a_{2i, d-1}$ intersects only the curve $a_{2i-1, d-1}$, and $i(a_{2i,d-1},a_{2i-1,d-1}) = 1$. It follows that if $a_{2i,d-1} \cap A_i$ has multiple components, exactly one is essential, and the remaining components can be isotoped off of $A_i$.

 In the general case where $a_{2i,j}$ intersects both $a_{2i-1,j}$ and $a_{2i-1,j+1}$, an analogous argument shows that $a_{2i,j} \cap A_i$ consists of one or two essential arcs. Consider the triangle in the L\"onne graph determined by $a_{2i,j}, a_{2i-1,j}, a_{2i-1, j+1}$. According to Lemma \ref{lemma:imgprops}.\ref{item:triangle}, the union $a_{2i,j}\cup a_{2i-1,j}\cup a_{2i-1, j+1}$ is supported on an essential subsurface of the form $\Sigma_{1,2}$. Figure \ref{figure:twoarc} shows that if $a_{2i,j} \cap A_i$ consists of two essential arcs, then $a_{2i,j}\cup a_{2i-1,j}\cup a_{2i-1, j+1}$ is supported on an essential subsurface of the form $\Sigma_{1,3}$, in contradiction with Lemma \ref{lemma:imgprops}.\ref{item:triangle}. Similar arguments establish that $a_{2i-2,j} \cap A_{i}$ is a single essential arc as well. 

\begin{figure}[ht!]
\begin{center}
\includegraphics{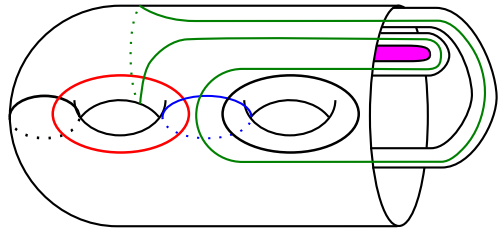}
\caption{If $a_{2i,j}$ cannot be isotoped onto a single arc inside $A_i$, then the curve enclosed by the inner strip (shaded) is essential in $\Sigma_g$, causing $a_{2i,j}\cup a_{2i-1,j}\cup a_{2i-1, j+1}$ to be supported on a surface $\Sigma_{1,3}$.}
\label{figure:twoarc}
\end{center}
\end{figure}

We next show that all points of intersection $a_{2i,j} \cap a_{2i, j +1}$ can be isotoped to occur on both $A_i$ and $A_{i+1}$. This also follows from Lemma \ref{lemma:imgprops}.\ref{item:triangle}. If some point of intersection $a_{2i,j} \cap a_{2i, j +1}$ could not be isotoped onto $A_i$, then the union  $a_{2i, j} \cup a_{2i, j+1} \cup a_{2i-1, j+1}$ could not be supported on a subsurface homeomorphic to $\Sigma_{1,2}$. An analogous argument applies with $A_{i+1}$ in place of $A_i$. This is explained in Figure \ref{figure:arcint}.

\begin{figure}[ht!]
\begin{center}
\includegraphics{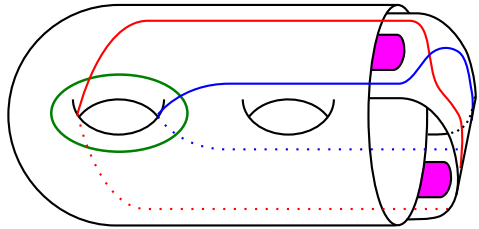}
\caption{If the intersection $a_{2i,j} \cap a_{2i,j+1}$ cannot be isotoped to occur on $A_i$, then both curves indicated by the shaded regions are essential in $\Sigma_g$, causing $a_{2i,j}\cup a_{2i,j+1}\cup a_{2i-1, j+1}$ to be supported on a surface $\Sigma_{1,3}$.}
\label{figure:arcint}
\end{center}
\end{figure}
It follows from this analysis that all crossings between curves in row $2i$ can be isotoped to occur in a collar neighborhood of $\partial A_i$. We define $A_i^+$ to be a slight enlargement of $A_i$ along such a neighborhood, so that all crossings between curves in row $2i$ occur in $A_i^+ \setminus A_i$. 

We can now understand what the collection of arcs $a_{2i,1} \cap A_i^+, \dots, a_{2i, d-1}\cap A_i^+$ looks like. To begin with, the change-of-coordinates principle asserts that up to a homeomorphism of $A_i$ fixing the curves $\{a_{2i-1,j}\}$, the arc $a_{2i,1} \cap A_i$ can be drawn in one of two ways. The first possibility is shown in Figure \ref{figure:arcsonAi}(a), and the second is its mirror-image obtained by reflection through the plane of the page (i.e. the curve with the dotted and solid portions exchanged). In fact, $a_{2i,1} \cap A_i$ must look as shown. This follows from Lemma \ref{lemma:imgprops}.\ref{item:triangle}. The vertices $(a_{2i-1, 1}, a_{2i-1, 2}, a_{2i,1})$ form a positively-oriented triangle, and so $i(a_{2i-1,1}, T_{a_{2i-1,2}}^{-1}(a_{2i,1})) = 0$. This condition precludes the other possibility. 

The pictures for $a_{2i,2}, \dots, a_{2i, d-1}$ are obtained by proceeding inductively. In each case, there are exactly two ways to draw an arc satisfying the requisite intersection properties, and Lemma \ref{lemma:imgprops}.\ref{item:triangle} precludes one of these possibilities. The result is shown in Figure \ref{figure:arcsonAi}(b). 

It remains to understand how the crossings between curves in row $2i$ are organized on $A_i^+$. As shown, the arcs $a_{2i,j}\cap A_i$ and $a_{2i,j+1}\cap A_i$ intersect $\partial A_i$ twice each, and in both instances the intersections are adjacent relative to the other arcs. There are thus apparently two possibilities for where the crossing can occur. However, one can see from Figure \ref{figure:arcsonAi}(c) that once a choice is made for one crossing, this enforces choices for the remaining crossings. Moreover, the two apparently distinct configurations are in fact equivalent: the cyclic ordering of the arcs along $\partial A_i^+$ is the same in either case, and the combinatorial type of the cut-up surface 
\[
A_i^\circ: = A_i^+ \setminus \bigcup \{a_{k, j} \mid 2i-1 \le k \le 2i, 1 \le j \le d-1 \}
\]
is the same in either situation. The change-of-coordinates principle then asserts the existence of a homeomorphism of $A_i^+$ sending each $a_{2i-1,j}$ to itself, and taking one configuration of arcs to the other. 

Having seen that the arcs $a_{2i,j} \cap A_i^+$ can be put into standard form, it remains to examine the other collection of arcs on $A_i^+$, namely those of the form $a_{2i-2,j}$. It is easy to see by induction on $d$ that the cut-up surface $A_i^\circ$ is a union of polyhedral disks for which the edges correspond to portions of the curves $a_{2i-1,j}$, the arcs $a_{2i,j} \cap A_i^+$, or else the boundary $\partial A_i^+$. It follows that the isotopy class of an arc $a_{2i-2,j}\cap A_I^+$ is uniquely determined by its intersection data with the curves $a_{2i-1,j}$ and $a_{2i,j}$.

 For $j \ge 2$, the curve $a_{2i-2,j}$ intersects $a_{2i-1,j-1}$ and $a_{2i-1,j}$, and is disjoint from all curves $a_{2i, k}$. As $a_{2i,j-1}$ has the same set of intersections as $a_{2i-2,j}$, it follows that $a_{2i-2,j} \cap A_i^+$ must run parallel to $a_{2i,j-1}$. The curve $a_{2i-2,1}$ intersects only $a_{2i,1}$; consequently $a_{2i-2,1}\cap A_i^+$ is uniquely determined. As can be seen from Figure \ref{figure:arcsonAi}(c), this forces each subsequent $a_{2i-2,j}$ onto a particular side of $a_{2i,j-1}$. 
 \begin{figure}[ht!]
 \begin{center}
 \labellist
 \small
 \pinlabel (a) at 161.6 104
\pinlabel (b) at 404 104
\pinlabel (c) at 295 4
 \endlabellist
 \includegraphics{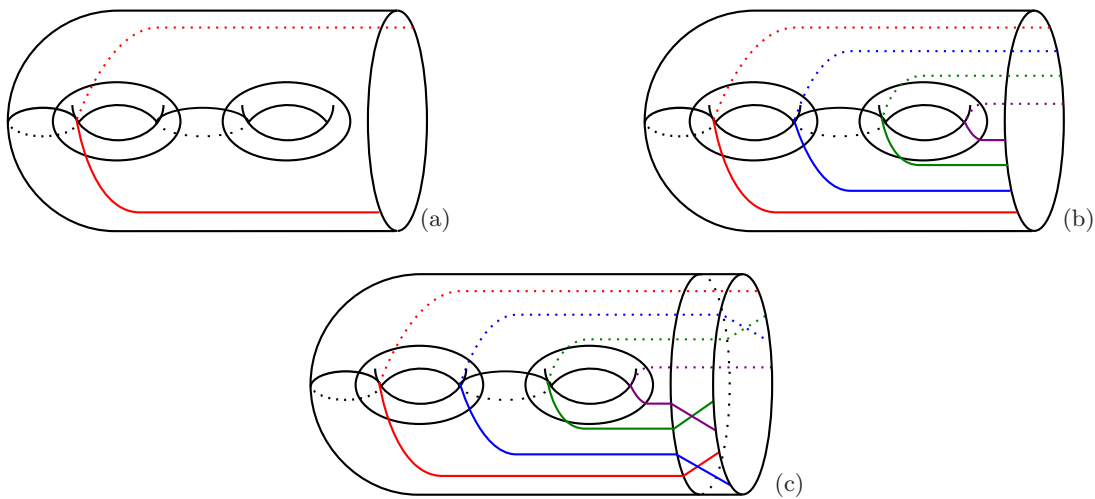}
 \caption{The surface $A_i^+$. (a): The correct choice for $a_{2i,1} \cap A_i$. (b): The configuration $a_{2i,j} \cap A_i$. (c) The configuration $a_{2i,j} \cap A_i^+$. }
 \label{figure:arcsonAi}
 \end{center}
 \end{figure}

\para{Step 3: Arcs on the remainder of $\Sigma_g$} Consider now the subsurface 
\[
\Sigma_g^\circ := \Sigma_g \setminus \bigcup A_i.
\] 
This has $(d-1)/2$ boundary components $\partial_{k}$, indexed by the corresponding $A_k$. The intersection $a_{2i,j} \cap \left( \Sigma_g \setminus \bigcup A_i\right)$ consists of two arcs, each connecting $\partial_{i}$ with $\partial_{i+1}$. The strategy for the remainder of the proof is to argue that when all these arcs are deleted from $\Sigma_g^\circ$, the result is a union of disks. The change-of-coordinates principle will then assert the uniqueness of such a configuration of arcs, completing the proof.

For what follows, it will be convenient to refer to a product neighborhood $[0,1] \times [0,1] \subset \Sigma_g^\circ$ of some arc $a_{2i,j} \cap \Sigma_g^\circ$ as a {\em strip}. Our first objective is to compute the Euler characteristic $\chi$ of the surface $\Sigma_g^{\circ \circ}$ obtained by deleting strips for all arcs from $\Sigma_g^\circ$. Then an analysis of the pattern by which strips are attached will determine the number of components of this surface.

To begin, we return to the setting of Figure \ref{figure:arcint}. Above, it was argued that for $2i < (d-1)/2$, the intersection $a_{2i,j} \cap a_{2i,j+1}$ can be isotoped onto either $A_{i}$ or $A_{i+1}$. This means that there is a strip that contains both $a_{2i,j} \cap \Sigma_g^\circ$ and $a_{2i,j+1} \cap \Sigma_g^\circ$. Grouping such strips together, it can be seen that for $1\le i \le (d-3)/2$, the $2i^{th}$ row of the L\"onne graph gives rise to $d$ strips. In the last row, there are $d-1$ strips. So in total there are $1/2(d+1)(d-2)$ strips, and each strip contributes $-1$ to the Euler characteristic. 

Recall the relation $g = (d-1)(d-2)/2$: this means that 
\[
\chi(\Sigma_g) = 2 - (d-1)(d-2).
\]
Each $A_i$ has Euler characteristic $\chi(A_i) = 2-d$. It follows that
\begin{eqnarray*}
\chi(\Sigma_g^\circ) = \chi(\Sigma_g) - \sum_{i = 1}^{(d-1)/2} \chi(A_i) =& 2 - (d-1)(d-2) + (d-1)(d-2)/2\\  
=& 2 - (d-1)(d-2)/2.
\end{eqnarray*}
Therefore,
\begin{eqnarray*}
\chi(\Sigma_g^{\circ \circ}) =& \chi(\Sigma_g^\circ) + 1/2(d+1)(d-2) \\ =& d.
\end{eqnarray*}

We claim that $\Sigma_g^{\circ \circ}$ has $d$ boundary components. This will finish the proof, as a surface of Euler characteristic $d$ and $b = d$ boundary components must be a union of $d$ disks. The claim can easily be checked directly in the case $d = 5$ of immediate relevance. For general $d$, this follows from a straightforward, if notationally tedious, verification, proceeding by an analysis of the cyclic ordering of the arcs $a_{i,j}$ around the boundary components $\partial A_k^+$. 
\end{proof}

\para{Step 2: A convenient configuration} 

\begin{figure}[ht!]
\begin{center}
\labellist
\small
\pinlabel $a_1$ at 28.8 262
\pinlabel $a_2$ at 86.4 228
\pinlabel $a_3$ at 76.8 157.6
\pinlabel $a_5$ at 158.4 261.6
\pinlabel $a_7$ at 172 121.2
\pinlabel $a_9$ at 167.2 342
\pinlabel $a_{10}$ at 240 232
\pinlabel $a_{11}$ at 231.2 114
\pinlabel $a_{13}$ at 306 288
\pinlabel $a_{15}$ at 308 184.8
\pinlabel $\iota$ at 8 248.8
\pinlabel $a_{15}$ at 303.2 112
\pinlabel $a_5$ at 335.2 112
\pinlabel $a_3$ at 369 112
\pinlabel $a_1$ at 403 112
\pinlabel $a_{13}$ at 303.2 78.4
\pinlabel $a_7$ at 335.2 78.4
\pinlabel $a_2$ at 369 78.4
\pinlabel $a_4$ at 403 78.4
\pinlabel $a_{11}$ at 303.2 44.8
\pinlabel $a_9$ at 335.2 44.8
\pinlabel $a_8$ at 369 44.8
\pinlabel $a_6$ at 403 44.8
\pinlabel $a_{10}$ at 303.2 11.2
\pinlabel $a_{12}$ at 335.2 11.2
\pinlabel $a_{14}$ at 369 11.2
\pinlabel $a_{16}$ at 403 11.2
\endlabellist
\includegraphics{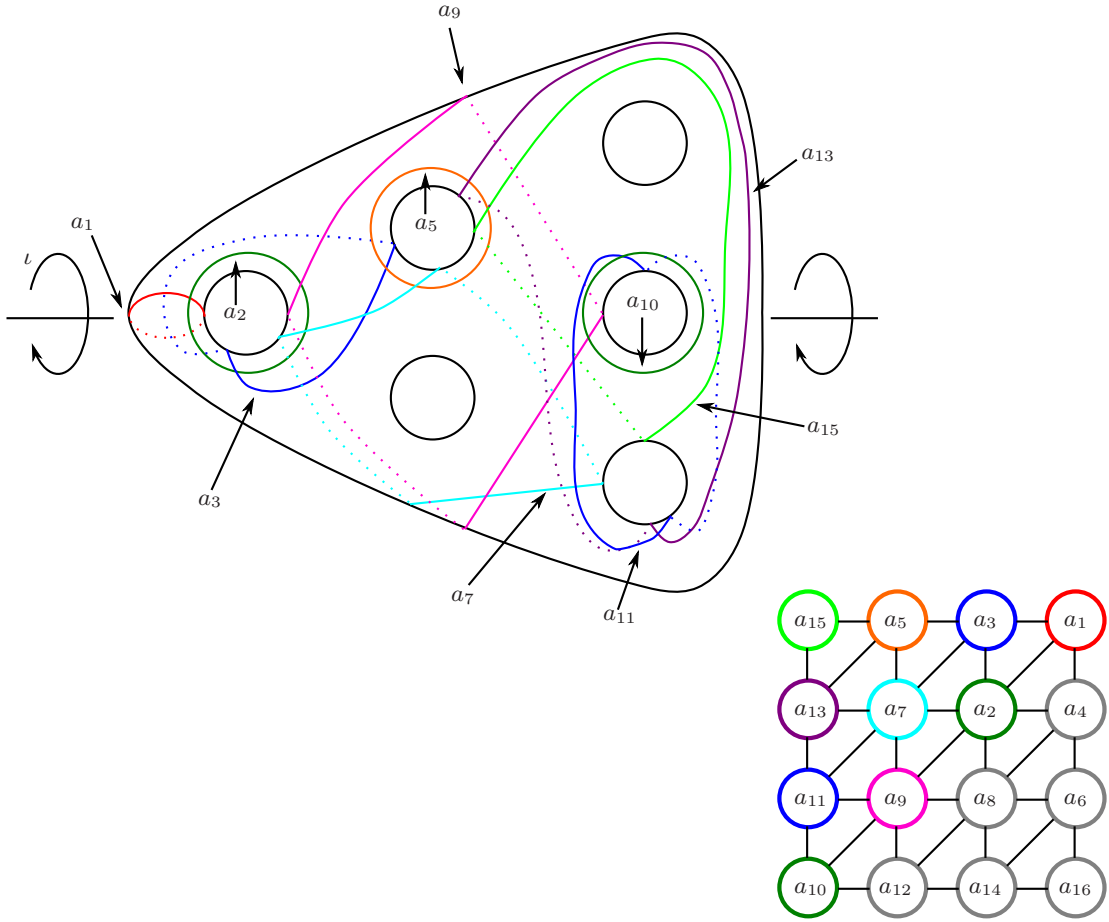}
\caption{A L\"onne configuration on $\Sigma_6$. Only a portion of the figure has been drawn: the omitted curves are obtained by applying the involution $\iota$ to the depicted curves.}
\label{figure:Lconfig}
\end{center}
\end{figure}

Figure \ref{figure:Lconfig} presents a picture of a L\"onne configuration in the case of interest $d = 5$. This was obtained by ``building the surface'' curve by curve, attaching one-handles in the sequence indicated by the numbering of the curves $a_1, \dots, a_{16}$. There are other, more uniform depictions of L\"onne configurations which arise from Akbulut-Kirby's picture of a plane curve of degree $d$ derived from a Seifert surface of the $(d,d)$ torus link (see \cite{AK} or \cite[Section 6.2.7]{GS}), but the analysis to follow is easier to carry out using the model of Figure \ref{figure:Lconfig}.

\para{Step 3: Producing vanishing cycles} The bulk of this step will establish claim (\ref{item:VC}); claim (\ref{item:square}) follows as an immediate porism. The principle is to exploit the fact that if $a$ and $b$ are vanishing cycles, then so is $T_a(b)$. To begin with, curves $c_1,c_2, c_4, c_8,$ and $c_{12}$ are elements of the L\"onne configuration and so are already vanishing cycles. The curve $c_3$ is obtained as
\[
c_3 = T_{a_2}^{-1}(a_3);
\]
similarly,
\[
c_{13} = T_{a_2}^{-1}(a_4).
\]
Curve $c_{10}$ is obtained as
\[
c_{10} = T_{a_{15}}(a_{13});
\]
$c_6$ is obtained from $a_{14}$ and $a_{16}$ analogously. 

The curve $c_9$ is obtained as 
\[
c_9 = T_{c_{10}} T_{a_{10}}^{-1} (a_{11});
\]
$c_7$ is obtained from $a_{10}, a_{12},$ and $c_6$ analogously.

To obtain $c_5$, twist $a_{13}$ along the chain $c_6, \dots, c_{10}$:
\[
c_5 = T_{c_6}^{-1} T_{c_7}^{-1} T_{c_8}^{-1} T_{c_9}^{-1} T_{c_{10}}^{-1}(a_{13}).
\]
$c_{11}$ is obtained by an analogous procedure on $a_{14}$. 

The sequence of twists used to exhibit $x$ as a vanishing cycle is illustrated in Figure \ref{figure:x}. Symbolically,
\[
x = T_{c_6}^{-1} T_{c_7}^{-1} T_{c_8}^{-1}T_{c_9}^{-1}T_{c_5} T_{c_4} T_{c_6}^{-1}T_{c_7}^{-1}T_{c_8}^{-1}T_{a_{9}}^{-1}(a_7).
\]
$y$ is produced in an analogous fashion, starting with $a_{8}$ in place of $a_6$.

To produce $z$, we appeal to the genus-$2$ star relation. Applied to the surface bounded by $b,y,z$, it shows that $T_b^2 T_y T_z \in \im(\rho_5)$, and hence $T_{b}^2 T_z \in \im(\rho_5)$ since $T_y \in \im(\rho_5)$ by above. Observe that $i(c_{10}, z) = 1$, and that $T_{c_{10}} \in \im(\rho_5)$. Making use of the fact that $b$ is disjoint from both $z$ and $c_{10}$, the braid relation gives
\[
T_{c_{10}}T_b^2 T_z(c_{10}) = T_{c_{10}} T_z (c_{10}) = z.
\]
This exhibits $z$ as a vanishing cycle, establishing claim (\ref{item:VC}) of Lemma \ref{lemma:config}. As $T_b^2 T_z$ and $T_z$ are now both known to be elements of $\im(\rho_5)$, it follows that $T_b^2 \in \im(\rho_5)$ as well, completing claim (\ref{item:square}). 

\begin{figure}[ht!]
\begin{center}
\labellist
\small
\pinlabel $T_{a_9}^{-1}$ at 120 180
\pinlabel $T_{c_6}^{-1}T_{c_7}^{-1}T_{c_8}^{-1}$ at 280 180
\pinlabel $T_{c_5}T_{c_4}$ at 236.8 92
\pinlabel $T_{c_6}^{-1}T_{c_7}^{-1}T_{c_8}^{-1}T_{c_9}^{-1}$ at 225 20
\endlabellist
\includegraphics{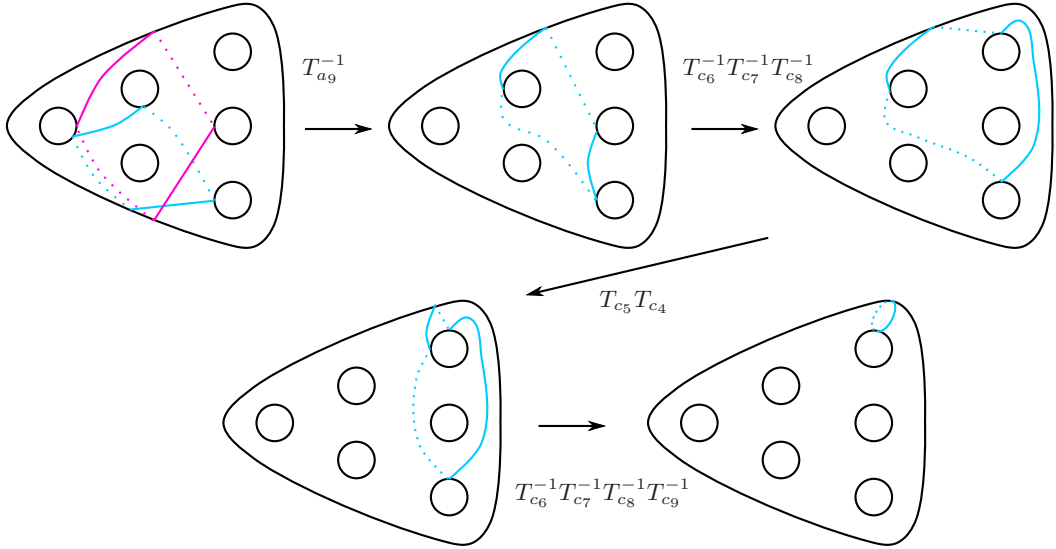}
\caption{The sequence of twists used to obtain $x$.}
\label{figure:x}
\end{center}
\end{figure}
\end{proof}

\section{Proof of Theorem \ref{theorem:d5}}\label{section:proofB}
	In this final section we assemble the work we have done so far in order to prove Theorem \ref{theorem:d5}.
	
	\para{Step 1: Reduction to the Torelli group} The first step is to reduce the problem of determining $\im(\rho_5)$ to the determination of $\im(\rho_5) \cap \mathcal{I}_6$. This will follow from \cite{beauville}. Recall that Beauville establishes that $\im(\Psi \circ \rho_5)$ is the entire stabilizer of an odd-parity spin structure on $H_1(\Sigma_6; \Z)$. This spin structure was identified as $\phi_5$ in Section \ref{section:construction}. Therefore $\im(\Psi \circ \rho_5) = \im(\Psi \circ \operatorname{Mod}(\Sigma_6)[\phi_5])$. It therefore suffices to show that
	\begin{equation}\label{equation:goal}
	\im(\rho_5) \cap \ker \Psi = \operatorname{Mod}(\Sigma_6)[\phi_5] \cap \ker \Psi = \mathcal{I}_6.
	\end{equation}
	
	\begin{figure}[ht!]
	\begin{center}
	\labellist
	\small
	\pinlabel $k=3$ at 0 358.4
\pinlabel $k=5$ at 0 244
\pinlabel $k=7$ at 0 128.4
\pinlabel $k=9$ at 0 0
\pinlabel $k=6$ at 163.2 358.4
\pinlabel $k=8$ at 163.2 244
\pinlabel $k=10$ at 163.2 128.4
\pinlabel $k=12$ at 163.2 0
\pinlabel $\gamma$ at 250 460
\endlabellist
	\includegraphics{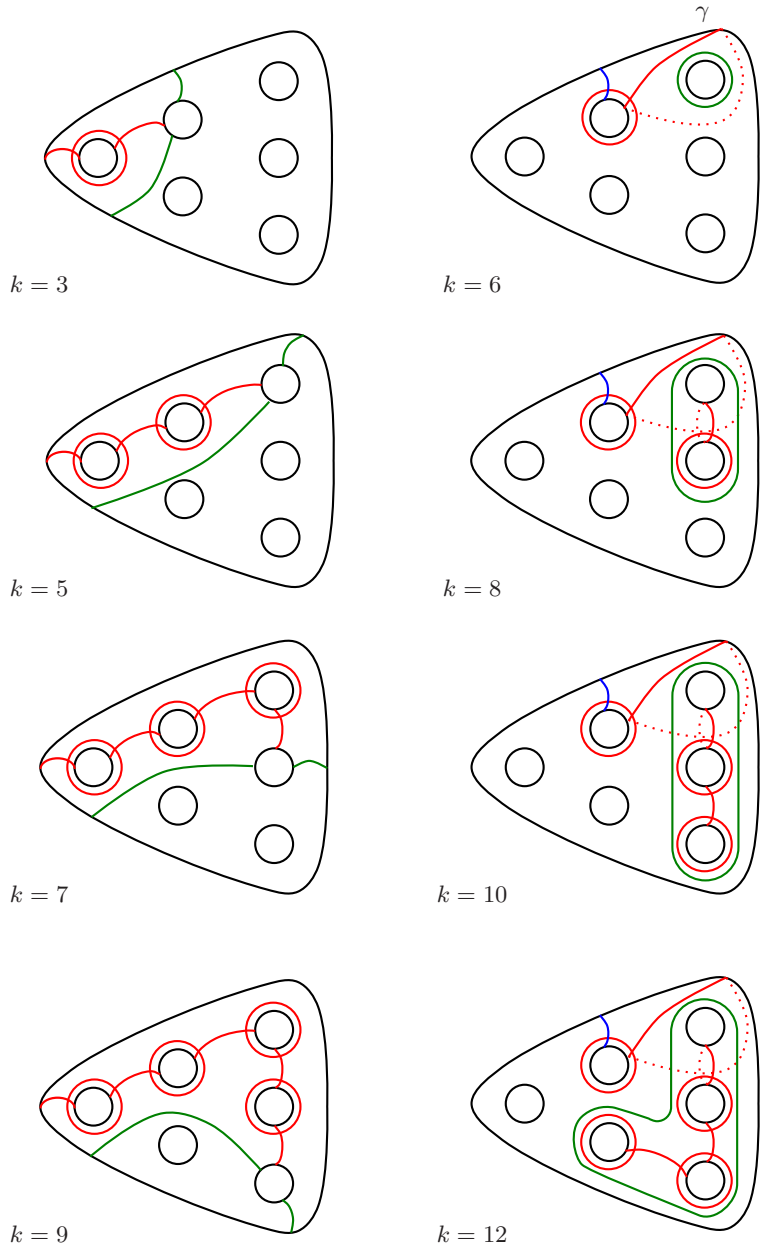}
	\caption{The cases of Step 2}
	\label{figure:cases}
	\end{center}
	\end{figure}
	
	\para{Step 2: Enumeration of cases} Equation (\ref{equation:goal}) will be derived as a consequence of Theorem \ref{theorem:genset}. Lemma \ref{lemma:config}.\ref{item:VC} asserts that the curves $c_1, \dots, c_{12}$ in the Johnson generating set are contained in $\im(\rho_5)$, so that the first hypothesis of Theorem \ref{theorem:genset} is satisfied. There are then eight cases to check: the four {\em straight chain maps} of the form $(c_1, \dots, c_k)$ for $k = 3,5,7,9$, and the four {\em $\beta$-chain maps} of the form $(\beta, c_5, \dots, c_k)$ for $k = 6, 8, 10, 12$. See Figure \ref{figure:cases}.
	
	The verification of the $\beta$-chain cases will be easier to accomplish after conjugating by the class $g = T_x T_{c_5}^{-1} T_{c_4}^{-1} \in \im(\rho_5)$. This has the following effect on the curves in the $\beta$-chains (the curve $\gamma$ is indicated in Figure \ref{figure:cases} in the picture for $k = 6$):
	\[
	g(\beta) = b, \qquad g(c_5) = c_4, \qquad g(c_6) = \gamma, \qquad g(c_k) = c_k \mbox{ for } k \ge 7.
	\]

	\para{Step 3: Producing bounding-pair maps} In this step, we explain the method by which we will obtain the necessary bounding-pair maps. This is an easy consequence of the chain relation. 
	\begin{lemma} \label{lemma:BP}
	Let $C = (c_1, \dots, c_k)$ be a chain of odd length $k$ and boundary $\partial C = d_1 \cup d_2$. Suppose that the mapping classes
	\[
	T_{c_1}^2, T_{c_2}, \dots, T_{c_k}, T_{d_1}^2
	\]
	are all contained in some subgroup $\Gamma \le \Mod(\Sigma_g)$. Then the chain map associated to $C$ (i.e. the bounding pair map $T_{d_1} T_{d_2}^{-1}$) is also contained in $\Gamma$. 
	\end{lemma}
	\begin{proof}
	The chain relation (Proposition \ref{proposition:chain}) implies that $T_{d_1} T_{d_2} \in \Gamma$. By hypothesis, $T_{d_1}^2 \in \Gamma$, so the bounding pair map $T_{d_1}T_{d_2}^{-1} \in \Gamma$ as well. 
	\end{proof}

	\para{Step 4: Verification of cases} Lemma \ref{lemma:config} asserts that the classes $T_{c_i}, 1 \le i \le 12$, as well as $T_b^2$ are all contained in $\im(\rho_5)$. The class $\gamma$ is obtained from $c_6$ by the element $g \in \im(\rho_5)$, so $\gamma$ is a vanishing cycle as well.  Via Lemma \ref{lemma:BP}, it remains only to show that in each of the cases in Step 2, one of the boundary components $d_1$ satisfies $T_{d_1}^2 \in \im(\rho_5)$. 
	
	The straight chain maps are depicted in the left-hand column of Figure \ref{figure:cases}. For $k = 3$, one boundary component is $b$; we have already remarked how $T_b^2 \in \im(\rho_5)$. For $k = 5$, one of the boundary components is $x$. For $k = 7$, one uses the methods of Lemma \ref{lemma:config} to show that the right-hand boundary component $c$ satisfies $T_c^2 \in \im(\rho_5)$ (the proof is identical to that for $b$). Finally, for $k = 9$, one of the boundary components is $y$.
	
	We turn to the $\beta$-chains. The images of the $\beta$-chains under the map $g$ are depicted in the right-hand column of Figure \ref{figure:cases}. For $k = 6,8,10,12$, let $d_k$ denote the boundary component depicted there for the chain $(b, c_4, \gamma, c_7, \dots, c_k)$. Observe that $d_k$ is also a boundary component of the chain map for $(c_6, \dots, c_k)$ (in the case $k=6$, the boundary component $d_6$ is just $c_6$). Moreover, the chain map for $(c_6, \dots, c_k)$ is conjugate to the chain map for $(c_1, \dots, c_{k-5})$ by an element of $\im(\rho_5)$ (this is easy to see using the isomorphism between the group generated by $c_1, \dots, c_{12}$ and the braid group $B_{13}$ on $13$ strands). Via the verification of the straight-chain cases, it follows that $T_{d_k}^2 \in \im(\rho_5)$, and so by Lemma \ref{lemma:BP} the $\beta$-chain maps are also contained in $\im(\rho_5)$. \qed

\bibliography{planecurvemonodromy}{}
\bibliographystyle{alpha}

\end{document}